\documentclass[a4paper,11pt,reqno]{amsart}

\usepackage[utf8]{inputenc}
\usepackage{amsmath,amssymb,amsthm,mathtools}
\usepackage[shortlabels]{enumitem}
\usepackage[left=3cm,right=3cm]{geometry}
\usepackage{hyperref}
\usepackage{xcolor}
\usepackage{microtype}
\usepackage{alltt}
\usepackage{stmaryrd}

\DeclareUnicodeCharacter{2016}{\textnormal{\ensuremath{\Vert}}}
\DeclareUnicodeCharacter{2102}{\textnormal{\ensuremath{\mathbb{C}}}}
\DeclareUnicodeCharacter{211D}{\textnormal{\ensuremath{\mathbb{R}}}}
\DeclareUnicodeCharacter{2115}{\textnormal{\ensuremath{\mathbb{N}}}}
\DeclareUnicodeCharacter{2192}{\textnormal{\ensuremath{\rightarrow}}}
\DeclareUnicodeCharacter{21A6}{\textnormal{\ensuremath{\mapsto}}}
\DeclareUnicodeCharacter{2200}{\textnormal{\ensuremath{\forall}}}
\DeclareUnicodeCharacter{2203}{\textnormal{\ensuremath{\exists}}}
\DeclareUnicodeCharacter{221E}{\textnormal{\ensuremath{\infty}}}
\DeclareUnicodeCharacter{2227}{\textnormal{\ensuremath{\wedge}}}
\DeclareUnicodeCharacter{2264}{\textnormal{\ensuremath{\leq}}}
\DeclareUnicodeCharacter{2265}{\textnormal{\ensuremath{\geq}}}
\DeclareUnicodeCharacter{2A05}{\textnormal{\ensuremath{\bigsqcap}}}

\newenvironment{leancode}%
  {\par\addvspace{6pt}\small
   \setlength{\topsep}{0pt}\setlength{\partopsep}{0pt}\begin{alltt}}%
  {\end{alltt}\addvspace{6pt}}
\newcommand{\lbr}{\symbol{123}}
\newcommand{\rbr}{\symbol{125}}
\newcommand{\lean}[1]{\texttt{\small #1}}

\numberwithin{equation}{section}

\hypersetup{colorlinks=true,linkcolor=blue,citecolor=red,urlcolor=black}

\theoremstyle{plain}
\newtheorem{theorem}{Theorem}[section]

\newtheorem{lemma}[theorem]{Lemma}
\newtheorem{corollary}[theorem]{Corollary}

\theoremstyle{definition}

\theoremstyle{remark}

\newcommand{\C}{\mathbb{C}}
\newcommand{\R}{\mathbb{R}}
\newcommand{\abs}[1]{\lvert #1\rvert}
\newcommand{\norm}[1]{\lVert #1\rVert}
\newcommand{\dd}{\,\mathrm{d}}
\newcommand{\ee}{\mathrm{e}}
\newcommand{\ii}{\mathrm{i}}
\newcommand{\Fock}{\mathcal{F}^2(\C)}
\newcommand{\dga}{\mathrm{d}\gamma}
\newcommand{\dgak}[1]{\mathrm{d}\gamma_{#1}}

\title[Complex analytic proof of square Fock SPR]{A complex-analytic proof of square-restricted stable phase retrieval in Fock space}
\author{Cynthia Bortolotto}
\author{Jo\~ao P.~G.~Ramos}

\begin{document}

\begin{abstract}
We give a short complex-analytic proof of a square-restricted form of local stable phase retrieval at the Gaussian in one-dimensional Fock space. The main estimate is a coercivity inequality for the map $F\mapsto F^2$:
\[
 \norm{F^2-F(0)^2}_{\Fock}
 \lesssim \inf_{c\in\R}\norm{\abs{F}^2-c}_{\Fock}.
\]
The proof uses a weighted derivative norm, two integrations by parts, and a weighted Cauchy inequality. The proof has been completely verified in Lean with the aid of Large Language Models. 
\end{abstract}

\maketitle

\section{Introduction}

Given a window function $g \in L^2(\R),$ we define the Short-time Fourier transform (STFT) of another $f \in L^2(\R)$ as 
\[
 V_g f(x,\omega)=\int_{\R} f(t)\overline{g(t-x)}\ee^{-2\pi\ii t\omega}\dd t.
\]
This transform is important in several different contexts, from signal processing to x-ray imaging \cite{Folland1989,Groechenig2001,Pfeiffer2018,Rodenburg2008}, and it is the main object of study of this paper, as well as one of the cornerstone objects of modern time-frequency analysis. As one major mathematical property of such object, we highlight that it is an isometry up to the norm of the window,
\[
 \norm{V_gf}_{L^2(\R^2)}=\norm{g}_{L^2(\R)}\norm{f}_{L^2(\R)},
\]
which allows such operators to retain many of the metric properties of $L^2(\R)$, with the major advantage that one often \emph{gains regularity} when applying it \cite[Chapters~3 and~11]{Groechenig2001}.  

We shall be interested in the case of a specific window $g$, which perhaps best exemplifies this: if we take $g(t) = 2^{1/4} \ee^{-\pi t^2}$ to be the $L^2$-normalized Gaussian on $\R,$ the transform $V_gf$ is called the \emph{Gabor transform}, and it is intimately connected to the theory of \emph{entire functions} on the plane. Indeed, one sees that the Gabor transform gives rise naturally to the so-called Bargmann transform
\[
 \mathcal{B}f(z):=2^{1/4}\int_{\R}f(t)\,
 \ee^{2\pi tz-\pi t^2-\frac{\pi}{2}z^2}\dd t,
 \qquad z\in\C,
\]
which satisfies the pointwise identity
\[
 \abs{V_gf(x,-\omega)}
 =\ee^{-\frac{\pi}{2}\abs{z}^2}\abs{\mathcal{B}f(z)},
 \qquad z=x+\ii\omega;
\]
see \cite{Bargmann1961} and \cite[Section~3.4]{Groechenig2001}, whose conventions we follow. This operator, as it turns out, is an isometric isomorphism between $L^2(\R)$ and the Bargmann-Fock space of entire functions on $\C$, which is nothing but the space of entire $F$ with $\int_{\C}\abs{F(z)}^2\ee^{-\pi\abs z^2}\dd m(z)<\infty$, where $m$ denotes Lebesgue measure on $\C$. Throughout this paper we work with the equivalent normalization
\[
 \dga(z):=\frac{1}{\pi}\ee^{-\abs{z}^2}\dd m(z),
 \qquad
 \norm{H}_{\Fock}^2:=\int_{\C}\abs{H(z)}^2\dga(z),
\]
$\Fock$ denoting the space of entire functions of finite norm; the substitution $H(w):=F(w/\sqrt{\pi})$ is a unitary map from the space of the previous display onto $\Fock$. The normalized monomials $e_n(z)=z^n/\sqrt{n!}$ form an orthonormal basis of $\Fock$, so if $H(z)=\sum_{n\geq0}a_nz^n$, then
\[
 \norm{H}_{\Fock}^2=\sum_{n\geq0}\abs{a_n}^2n!.
\]
We refer to \cite{Zhu2012} for a systematic treatment of Fock spaces.

Such a window has, beyond its interesting analyticity properties, the practical advantage of yielding a lot of spatial concentration, due to the rapid decay of the Gaussian. In that regard, it is the perhaps most natural object in order to investigate one of the main mathematical problems related to the applied theory of the Short-time Fourier transform, which is that of phase retrieval. Phase retrieval asks whether an unknown function can be recovered, up to a constant phase, from the modulus of its STFT; equivalently, the inverse problem is to reconstruct $f$ from the phaseless measurement $\abs{V_gf}$, or from the intensity $\abs{V_gf}^2$. This problem is central in time--frequency analysis and also arises in ptychography, where an unknown object is probed by translates of a window and only far-field intensities are recorded.

If stability is ignored, phase retrieval can often be understood through the ambiguity function. With
\[
 \mathcal{A}f(x,\omega):=\ee^{\pi\ii x\omega}V_ff(x,\omega),
\]
one has the fundamental identity
\[
 \widehat{\abs{V_gf}^2}(\omega,-x)
 =\mathcal{A}f(x,\omega)\overline{\mathcal{A}g(x,\omega)}.
\]
For a Gaussian window, $\mathcal{A}g$ is a nonvanishing Gaussian, and this identity gives phase retrieval. For Hermite windows, the ambiguity function is a Gaussian times a Laguerre polynomial and therefore has zeros on finitely many circles. These zeros do not prevent phase retrieval, but they do make quantitative stability more delicate, since a direct inversion would require division by a function that degenerates along those circles. See \cite{GroechenigJamingMalinnikova2020} for a systematic study of zeros of ambiguity functions and STFTs.

On the complex-analytic side, the Hermite windows correspond to the scale of true polyanalytic Fock spaces, of which $\Fock$ is the first member; see, for example, \cite{AbreuFeichtinger2014}. For $k\geq0$, the $k$th space is generated by the true Hermite basis
\[
 \Phi_{k,n}(z)=\frac{1}{\sqrt{k!n!}}
 \sum_{j=0}^{\min(k,n)}(-1)^j\binom{k}{j}\frac{n!}{(n-j)!}
 z^{n-j}\overline z^{\,k-j},
 \qquad n\geq0.
\]
The case $k=0$ is exactly $\Fock$, whereas $k\geq1$ gives a genuinely polyanalytic space. The companion stability theorem discussed below treats these Hermite settings at canonical finite combinations of basis vectors. The present paper stays in the analytic case $k=0$, where the Cauchy--Riemann equations provide an additional identity that is unavailable in the polyanalytic setting.

The fundamental stability question is whether the distance to the global-phase orbit can be controlled by the distance between phaseless measurements. In the natural Hilbert-space formulation, one seeks an estimate of the form
\[
 \inf_{\abs{\lambda}=1}\norm{f-\lambda h}_{L^2(\R)}
 \leq C(f)\norm{\abs{V_g f}-\abs{V_g h}}_{L^2(\R^2)}.
\]

More generally, if $E$ is a closed subspace of some $L^2(\Omega,\mu)$, one can ask whether the modulus map is locally stable at $f\in E$:
\begin{equation}
 \inf_{\abs{\lambda}=1}\norm{f-\lambda h}_{L^2(\mu)}
 \leq C(f)\norm{\abs f-\abs h}_{L^2(\mu)},
 \qquad h\in E.
 \label{eq:abstract-stability-intro}
\end{equation}
This formulation serves as evidence as to why infinite-dimensional stability is subtle. Uniform Lipschitz bounds fail for phase retrieval by continuous frames in infinite-dimensional spaces \cite{CahillCasazzaDaubechies2016,AlaifariGrohs2017}, and even Gabor phase retrieval is severely ill-posed in its unrestricted form \cite{AlaifariGrohs2021}. More recently, it was shown that the local stability constant is infinite on a dense set in broad classes of phase-retrieval spaces \cite{AlharbiEtAl2024}.

These negative results, however, do not end the stability question, as they merely indicate that one must identify the geometry of the instability and the points at which it can be excluded. For Gaussian-window Gabor measurements, this led to quantitative estimates in which the stability constant is governed by connectivity or Poincar\'e-type quantities of the spectrogram \cite{AlaifariDaubechiesGrohsYin2019,GrohsRathmair2019,GrohsRathmair2022,Rathmair2024}. Such estimates make precise the principle that instability is caused by separated components or by mass escaping to infinity. They generally use a stronger first-order Sobolev measurement norm rather than the natural $L^2$ norm. Thus two basic issues remain visible: robustness under a change of window and stability in the topology in which the forward STFT map is naturally bounded.

In the Fock model the difficulty is purely metric: the modulus of an entire function determines it up to a unimodular constant, so uniqueness is immediate and only the quantitative question is left, and that question is genuinely infinite dimensional. The preceding literature provides powerful conditional stability estimates, but it does not by itself give local stability in the unmodified $L^2$ measurement norm at a concrete nonzero signal. The companion paper \cite{Bertolini2026}  is one of the first instances where a positive result which makes use of the full geometry was shown. Indeed, it proves the global estimate
\begin{equation}
 \inf_{\abs{\lambda}=1}\norm{F-\lambda}_{\Fock}
 \leq M\norm{\abs{F}-1}_{\Fock},
 \qquad F\in\Fock,
 \label{eq:full-result-intro}
\end{equation}
at the constant function. More generally, it establishes local stability for the canonical Hermite windows at elements of the finite span of their basis vectors. Its proof proceeds through an orthogonal reduction and a quantitative estimate for trigonometric polynomials on circles, followed by frequency localization on annuli; see \cite{Bertolini2026}.

The point of the present manuscript is to go a different way, providing a different argument in a particular context. It is shorter and uses only complex analysis, but it applies to a smaller, nonlinear but still infinite-dimensional class of competitors, namely the class of squares
\[
 \mathcal{S}:=\{F^2:F\text{ is entire and }F^2\in\Fock\}.
\]
Our main estimate is the following.

\begin{theorem}\label{thm:main}
There is an absolute constant $C>0$ such that, for every polynomial $F$,
\begin{equation}
 \norm{F^2-F(0)^2}_{\Fock}
 \leq C\inf_{c\in\R}\norm{\abs{F}^2-c}_{\Fock}.
 \label{eq:square-coercivity}
\end{equation}
The same estimate holds for every entire $F$ such that $F^2\in\Fock$.
\end{theorem}

The theorem immediately gives a square-restricted version of \eqref{eq:full-result-intro}.

\begin{corollary}\label{cor:square-stability}
There is an absolute constant $C>0$ such that, for every entire $F$ with $F^2\in\Fock$,
\begin{equation}
 \inf_{\abs{\lambda}=1}\norm{F^2-\lambda}_{\Fock}
 \leq C\norm{\abs{F}^2-1}_{\Fock}.
 \label{eq:square-stability}
\end{equation}
\end{corollary}

The proof of Theorem~\ref{thm:main} has three ingredients. First, on the subspace of Fock functions vanishing at the origin, the Fock norm is equivalent to a weighted norm of the derivative, enabling a direct replacement thereof. Second, the Cauchy--Riemann equations identify the gradient of $\abs{F}^2$ with the derivative of $F^2$. Finally, two integrations by parts reduce the remaining expression to weighted $L^4$ norms of $F'$ and $F''$; a scale adapted Cauchy estimate controls the latter by the former.

The method is closely related in spirit to Gaussian energy methods. The classical Gaussian logarithmic Sobolev inequality originates in work of Gross \cite{Gross1975}; the complex-analytic viewpoint particularly relevant here appears in Carlen's work on Fisher information and logarithmic Sobolev inequalities \cite{Carlen1991}. 

Finally, the entire argument below has been formalized in Lean~4 and checked by its kernel. Section~\ref{sec:lean} states what was verified, matches it step by step with the proof, and records what the formalization required beyond the text.

\subsection*{LLM Usage} The ideas in this manuscript are completely devised by humans. A good part of their writing was itself done by humans, but several large language models, such as GPT 5.5 and 5.6 Sol, as well as Claude Opus 4.8, 5 and Fable 5 have been used to punctually expand computations, rewrite parts of sentences and proof check the manuscript. 

The main part where LLMs were instrumental was in the Lean formalization part. Through extensive use of Claude Opus 4.8, 5, Fable 5, and GPT 5.5 and 5.6 Sol, the manuscript has been thoroughly verified in Lean. 

In spite of the usage of LLMs, we have checked carefully all results in this manuscript, including (and, as a matter of fact, most carefully of all) the Lean statements. We have significantly changed parts where the LLMs edited the manuscript, and we have rewritten several other parts to more closely match our intuition. We take full responsibility for the contents of this paper. 

\section{Main argument}\label{sec:main-argument}

Throughout the proof, $m$ denotes Lebesgue measure on $\C$. For $k\geq0$, write
\[
 \mathrm{d}\gamma_k(z):=\frac{1}{\pi}\frac{\ee^{-\abs{z}^2}}{(1+\abs{z}^2)^k}\dd m(z),
\]
Thus $\dga=\dgak{0}$. We first prove that, when $H(0)=0$, the Fock norm of $H$ is equivalent to a weighted $L^2$-norm of $H'$.

\begin{lemma}\label{lem:derivative-equivalence}
There are absolute constants $0<c\leq C<\infty$ such that, whenever $H\in\Fock$ and $H(0)=0$,
\begin{equation}
 c\norm{H}_{\Fock}^2
 \leq \int_{\C}\abs{H'(z)}^2\dgak{1}(z)
 \leq C\norm{H}_{\Fock}^2.
 \label{eq:derivative-equivalence}
\end{equation}
\end{lemma}

\begin{proof}
Since $H(0)=0$, the Taylor expansion of $H$ has the form
$H(z)=\sum_{n\geq1}a_nz^n$. Note that
\[
 H'(r\ee^{\ii\theta})
 =\sum_{n\geq1}na_nr^{n-1}\ee^{\ii(n-1)\theta}.
\]
Changing to polar coordinates and using the orthogonality of
$\theta\mapsto\ee^{\ii k\theta}$ in the angular integral, we obtain
\begin{align*}
 \int_{\C}\abs{H'}^2\dgak{1}
 &=\frac{1}{\pi}\int_0^\infty\int_0^{2\pi}
   \abs{\sum_{n\geq1}na_nr^{n-1}\ee^{\ii(n-1)\theta}}^2
   \frac{\ee^{-r^2}}{1+r^2}r\dd\theta\dd r\\
 &=2\sum_{n\geq1}n^2\abs{a_n}^2
   \int_0^\infty\frac{r^{2n-1}\ee^{-r^2}}{1+r^2}\dd r\\
 &=\sum_{n\geq1}n^2\abs{a_n}^2\kappa_n,
\end{align*}
where
\[
 \kappa_n:=\int_{\C}\frac{\abs{z}^{2(n-1)}}{1+\abs{z}^2}\dga(z).
\]
Here the derivative series is uniformly convergent on each circle, and Tonelli's
theorem justifies the interchange with the resulting nonnegative radial
series. Observe that it suffices to prove that
\begin{equation}
 \frac{n^2\kappa_n}{n!}\sim1,
 \qquad n\geq1.
 \label{eq:kappa-comparison}
\end{equation}
Indeed, if \eqref{eq:kappa-comparison} holds, then, since
\(\norm{H}_{\Fock}^2=\sum_{n\geq1}\abs{a_n}^2n!\), the identity above can be
compared term by term, and summing over $n$ gives
\[
 \int_{\C}\abs{H'}^2\dgak{1}\sim\norm{H}_{\Fock}^2.
\]

Therefore, we only need to prove \eqref{eq:kappa-comparison}. To this end, using
$(1+\abs{z}^2)^{-1}=\int_0^\infty\ee^{-(1+\abs{z}^2)s}\dd s$ and
\(\int_{\C}\abs{z}^{2j}\dga=j!\), we obtain
\begin{equation}
 \frac{n^2\kappa_n}{n!}
 =n\int_0^\infty\frac{\ee^{-s}}{(1+s)^n}\dd s.
 \label{eq:kn}
\end{equation}
For the lower bound in \eqref{eq:kappa-comparison}, we use \eqref{eq:kn} and
restrict the integral to $0\leq s\leq1/n$.
On this interval, we use that $(1+s)^{-n} \ge e^{-ns}$, and therefore
\[
 n\int_0^\infty\frac{\ee^{-s}}{(1+s)^n}\dd s
 \geq n\int_0^{1/n}\ee^{-2}\dd s
 =\ee^{-2}.
\]
For the upper bound,  split the integral at $s=1$. On $0\leq s\leq1$ we use that
$\log(1+s)\geq s/2$, and on $s\geq1$ we have
$(1+s)^{-n}\leq2^{-n}$. Hence
\[
 \begin{aligned}
 n\int_0^\infty\frac{\ee^{-s}}{(1+s)^n}\dd s
 &\leq n\int_0^1\ee^{-ns/2}\dd s
   +n2^{-n}\int_1^\infty\ee^{-s}\dd s\\
 &\leq2+\frac{1}{2\ee}<3,
 \end{aligned}
\]
finishing the proof.
\end{proof}

We will also use the following elementary weighted Cauchy estimate.

\begin{lemma}\label{lem:weighted-cauchy}
There is an absolute constant $A$ such that every polynomial $F$ satisfies
\begin{align}
 \int_{\C}\abs{F''}^4\dgak{4}
 &\leq A\int_{\C}\abs{F'}^4\dgak{2},
 \label{eq:cauchy-fourth}\\
 \int_{\C}\abs{F'}^2\abs{F''}^2\dgak{3}
 &\leq A\int_{\C}\abs{F'}^4\dgak{2}.
 \label{eq:cauchy-mixed}
\end{align}
\end{lemma}

\begin{proof}
We first prove \eqref{eq:cauchy-fourth}. The Cauchy integral formula applied to
$F'$ gives
\[
 \abs{F''(z)}
 \leq\frac{1}{2\pi r}\int_0^{2\pi}
   \abs{F'(z+r\ee^{\ii\theta})}\dd\theta,
\]
and H\"older's inequality yields, for every $r>0$,
\begin{equation}
 \abs{F''(z)}^4
 \leq\frac{1}{2\pi r^4}\int_0^{2\pi}
 \abs{F'(z+r\ee^{\ii\theta})}^4\dd\theta.
 \label{eq:pointwise-cauchy}
\end{equation}

To compute the integral in \eqref{eq:cauchy-fourth}, we split the plane into
the regions $\{z:\abs z\leq2\}$ and $\{z:\abs z>2\}$.

Note that if $\abs z\leq2$, then
\[
 \frac{\ee^{-\abs z^2}}{(1+\abs z^2)^4}
 \leq C_1
   \frac{\ee^{-\abs{z+\ee^{\ii\theta}}^2}}
        {(1+\abs{z+\ee^{\ii\theta}}^2)^2}.
\]
Therefore, using \eqref{eq:pointwise-cauchy} with
$r=1$, Tonelli's theorem, and translation by $\ee^{\ii\theta}$, we obtain
\begin{align*}
 \int_{\{\abs z\leq2\}}\abs{F''(z)}^4\dgak{4}(z)
 &\leq\frac{1}{2\pi^2}\int_0^{2\pi}\int_{\{\abs z\leq2\}}
   \abs{F'(z+\ee^{\ii\theta})}^4
   \frac{\ee^{-\abs z^2}}{(1+\abs z^2)^4}
   \dd m(z)\dd\theta\\
 &\leq\frac{C_1}{2\pi^2}
   \int_0^{2\pi}\int_{\{\abs{\zeta-\ee^{\ii\theta}}\leq2\}}
   \abs{F'(\zeta)}^4
   \frac{\ee^{-\abs\zeta^2}}{(1+\abs\zeta^2)^2}
   \dd m(\zeta)\dd\theta\\
 &\leq C_1
   \int_{\C}\abs{F'}^4\dgak{2}.
\end{align*}

We now consider $\{z:\abs z>2\}$. Here a fixed radius would not give a
uniform comparison of the Gaussian weights, so we choose
$r_z=\abs z^{-1}$. For fixed $\theta$, set
\[
 \Phi_\theta(z):=z+\frac{\ee^{\ii\theta}}{\abs z}.
\]
Then \eqref{eq:pointwise-cauchy} becomes
\[
 \abs{F''(z)}^4
 \leq\frac{\abs z^4}{2\pi}\int_0^{2\pi}
   \abs{F'(\Phi_\theta(z))}^4\dd\theta.
\]

Before integrating, we compare the weight at $z$ with the weight at
$\Phi_\theta(z)$. The choice $r_z=\abs z^{-1}$ ensures that the difference
between $\abs{\Phi_\theta(z)}^2$ and $\abs z^2$ is uniformly bounded.
In particular,
\[
 \abs{\Phi_\theta(z)}^2-\abs z^2\leq\frac94,
 \qquad
 1+\abs{\Phi_\theta(z)}^2\leq2(1+\abs z^2),
\]
and therefore
\[
 \frac{\abs z^4\ee^{-\abs z^2}}{(1+\abs z^2)^4}
 \leq4\ee^{9/4}
   \frac{\ee^{-\abs{\Phi_\theta(z)}^2}}
        {(1+\abs{\Phi_\theta(z)}^2)^2}.
\]

Multiplying the pointwise estimate by the density of $\dgak{4}$, integrating,
and using the inequality above, we obtain
\begin{align*}
 \int_{\{\abs z>2\}}\abs{F''(z)}^4\dgak{4}(z)
 &\leq\frac{1}{2\pi^2}\int_0^{2\pi}\int_{\{\abs z>2\}}
   \abs{F'(\Phi_\theta(z))}^4
   \frac{\abs z^4\ee^{-\abs z^2}}{(1+\abs z^2)^4}
   \dd m(z)\dd\theta\\
 &\leq\frac{2\ee^{9/4}}{\pi^2}
   \int_0^{2\pi}\int_{\{\abs z>2\}}
   \abs{F'(\Phi_\theta(z))}^4
   \frac{\ee^{-\abs{\Phi_\theta(z)}^2}}
        {(1+\abs{\Phi_\theta(z)}^2)^2}
   \dd m(z)\dd\theta.
\end{align*}

To apply the change-of-variables formula to $\Phi_\theta$, we check that this
map is one-to-one and that its Jacobian determinant is bounded away from zero. The real differential of $\Phi_\theta$ is
\[
 D\Phi_\theta(z)[h]
 =h-\frac{\ee^{\ii\theta}}{\abs z^3}
   \operatorname{Re}(\overline zh),
\]
and hence
\[
 \det D\Phi_\theta(z)
 =1-\frac{\operatorname{Re}(\ee^{-\ii\theta}z)}{\abs z^3}
 \geq1-\frac{1}{\abs z^2}
 \geq\frac34.
\]
Moreover, for $\abs z>2$ and $\abs{z'}>2$,
\begin{align*}
 \abs{\Phi_\theta(z)-\Phi_\theta(z')}
 &\geq\abs{z-z'}
   -\left\lvert\frac{1}{\abs z}-\frac{1}{\abs{z'}}\right\rvert\\
 &\geq\frac34\abs{z-z'}.
\end{align*}
Thus $\Phi_\theta$ is one-to-one on $\{z:\abs z>2\}$. Since
$\det D\Phi_\theta(z)\geq3/4$, the change-of-variables formula gives
\begin{align*}
 \int_{\{\abs z>2\}}\abs{F'(\Phi_\theta(z))}^4
   \frac{\ee^{-\abs{\Phi_\theta(z)}^2}}
        {(1+\abs{\Phi_\theta(z)}^2)^2}\dd m(z)
 &\leq\frac43\int_{\C}
   \abs{F'(\zeta)}^4
   \frac{\ee^{-\abs\zeta^2}}{(1+\abs\zeta^2)^2}\dd m(\zeta)\\
 &=\frac{4\pi}{3}
   \int_{\C}\abs{F'}^4\dgak{2}.
\end{align*}
Consequently,
\[
 \int_{\{\abs z>2\}}\abs{F''(z)}^4\dgak{4}(z)
 \leq C_2
   \int_{\C}\abs{F'}^4\dgak{2},
\]
and putting the two estimates together, we get
\[
 \int_{\C}\abs{F''}^4\dgak{4}
 \leq A_4\int_{\C}\abs{F'}^4\dgak{2},
 \qquad
 A_4:=C_1+C_2,
\]
where $A_4$ is an absolute constant. This proves
\eqref{eq:cauchy-fourth}.

To prove \eqref{eq:cauchy-mixed}, we just apply Cauchy-Schwarz and get
\begin{align*}
 \int_{\C}\abs{F'}^2\abs{F''}^2\dgak{3}
 &=\frac{1}{\pi}\int_{\C}
   \left(\abs{F'}^2\frac{\ee^{-\abs z^2/2}}{1+\abs z^2}\right)
   \left(\abs{F''}^2\frac{\ee^{-\abs z^2/2}}{(1+\abs z^2)^2}\right)
   \dd m(z)\\
 &\leq
 \left(\int_{\C}\abs{F'}^4\dgak{2}\right)^{1/2}
 \left(\int_{\C}\abs{F''}^4\dgak{4}\right)^{1/2}\\
 &\leq A_4^{1/2}\int_{\C}\abs{F'}^4\dgak{2}.
\end{align*}
Taking $A:=\max\{A_4,A_4^{1/2}\}$ proves both estimates.
\end{proof}

We are now ready to prove Theorem~\ref{thm:main}.

\begin{proof}[Proof of Theorem~\ref{thm:main}]
Let $F$ be a polynomial and $c \in \mathbb{R}$. The
Cauchy--Riemann equations give
\[
 \partial_x\abs F^2=2\operatorname{Re}(F'\overline F),
 \qquad
 \partial_y\abs F^2=-2\operatorname{Im}(F'\overline F),
\]
and consequently, we have the pointwise identity
\begin{equation}
 \abs{\nabla(\abs F^2)}^2=4\abs F^2\abs{F'}^2=\abs{(F^2)'}^2.
 \label{eq:cr-identities}
\end{equation}
Since $F^2-F(0)^2$ vanishes at the origin, applying
Lemma~\ref{lem:derivative-equivalence} to $F^2-F(0)^2$ and then using
\eqref{eq:cr-identities} gives
\[
 \norm{F^2-F(0)^2}_{\Fock}^2
 \leq C\int_{\C}\abs{\nabla(\abs F^2)}^2\dgak{1}.
\]
To bound the integral above, fix $c\in\R$ and let
\[
 w_1(z):=\frac{1}{\pi}\frac{\ee^{-\abs z^2}}{1+\abs z^2}.
\]
Subtracting $c$ does not change the gradient, so
\[
 \int_{\C}\abs{\nabla(\abs F^2)}^2w_1\dd m
 =\int_{\C}\nabla(\abs F^2-c)\cdot\nabla(\abs F^2)w_1\dd m.
\]
Integrating by parts once gives
\begin{align*}
 \int_{\C}\abs{\nabla(\abs F^2)}^2w_1\dd m
 &=-\int_{\C}(\abs F^2-c)\,\Delta(\abs F^2)\,w_1\dd m\\
 &\quad-\int_{\C}(\abs F^2-c)\,
   \nabla(\abs F^2)\cdot\nabla w_1\dd m.
\end{align*}
To treat the second integral, first note that
\[
 (\abs F^2-c)\nabla(\abs F^2)
 =\frac12\nabla\left((\abs F^2-c)^2\right).
\]
We use this identity to rewrite the integral as
\[
 -\int_{\C}(\abs F^2-c)\,\nabla(\abs F^2)\cdot\nabla w_1\dd m
 =-\frac12\int_{\C}\nabla\left((\abs F^2-c)^2\right)
   \cdot\nabla w_1\dd m.
\]
We now integrate by parts to obtain
\[
 -\frac12\int_{\C}\nabla\left((\abs F^2-c)^2\right)
   \cdot\nabla w_1\dd m
 =\frac12\int_{\C}(\abs F^2-c)^2\Delta w_1\dd m.
\]
Combining the two identities, we obtain
\begin{equation}
 \int_{\C}\abs{\nabla(\abs F^2)}^2w_1\dd m
 =-\int_{\C}(\abs F^2-c)\,\Delta(\abs F^2)\,w_1\dd m
 +\frac12\int_{\C}(\abs F^2-c)^2\Delta w_1\dd m.
 \label{eq:first-ibp}
\end{equation}
The boundary terms vanish because $F$ is a polynomial and $w_1$ has Gaussian
decay.

The explicit form of $w_1$ gives
\[
 \abs{\Delta w_1(z)}\leq C\frac{\ee^{-\abs z^2}}{\pi},
\]
and therefore the last term in \eqref{eq:first-ibp} is bounded by
$C\norm{\abs F^2-c}_{\Fock}^2$.

For the first term on the right-hand side of \eqref{eq:first-ibp},
Cauchy--Schwarz yields
\[
 \left|\int_{\C}(\abs F^2-c)\,\Delta(\abs F^2)\,w_1\dd m\right|
 \leq \norm{\abs F^2-c}_{\Fock}J^{1/2},
 \qquad \text { where }
 J:=\int_{\C}\bigl(\Delta(\abs F^2)\bigr)^2\dgak{2}.
\]
Therefore, it remains to show that
$J\lesssim\norm{\abs F^2-c}_{\Fock}^2$ to conclude the proof.  Put
\[
 w_2(z):=\frac{1}{\pi}\frac{\ee^{-\abs z^2}}{(1+\abs z^2)^2},
\]
and observe that, since $\Delta(\abs F^2-c)=\Delta(\abs F^2)$, two further integrations by
parts give
\begin{align}
 J
 &=\int_{\C}(\abs F^2-c)\,
   \Delta\bigl(\Delta(\abs F^2)w_2\bigr)\dd m \notag\\
 &=\int_{\C}(\abs F^2-c)\,\Delta^2(\abs F^2)\,w_2\dd m \notag\\
 &\quad+2\int_{\C}(\abs F^2-c)\,
   \nabla\bigl(\Delta(\abs F^2)\bigr)\cdot\nabla w_2\dd m \notag\\
 &\quad+\int_{\C}(\abs F^2-c)\,\Delta(\abs F^2)\,\Delta w_2\dd m.
 \label{eq:J-expansion}
\end{align}
Note that we can bound the derivatives of $w_2$ by
\[
 \abs{\nabla w_2(z)}\leq C\frac{\ee^{-\abs z^2}}{\pi(1+\abs z^2)^{3/2}},
 \qquad
 \abs{\Delta w_2(z)}\leq C\frac{\ee^{-\abs z^2}}{\pi(1+\abs z^2)}.
\]
For any holomorphic function $G$, the Cauchy--Riemann equations give
$\Delta(\abs G^2)=4\abs{G'}^2$. Applying this identity first to $G=F$ and
then to $G=F'$ gives
\begin{equation}
 \Delta(\abs F^2)=4\abs{F'}^2,
 \qquad
 \Delta^2(\abs F^2)=4\Delta(\abs{F'}^2)=16\abs{F''}^2.
 \label{eq:laplacian-identities}
\end{equation}
Moreover, applying the pointwise identity \eqref{eq:cr-identities} with $F'$
in place of $F$ gives
\[
 \abs{\nabla\bigl(\Delta(\abs F^2)\bigr)}
 =4\abs{\nabla\abs{F'}^2}
 =8\abs{F'}\abs{F''}.
\]
The first identity in \eqref{eq:laplacian-identities} also shows that
\begin{equation}
 J=16\int_{\C}\abs{F'}^4\dgak{2}.
 \label{eq:j-f-prime}
\end{equation}

We now estimate separately the three integrals on the right-hand side of
\eqref{eq:J-expansion}. For the first integral, we use the second identity in
\eqref{eq:laplacian-identities}. Since $w_2\dd m=\dgak{2}$ and the density of
$\dgak{2}$ is the geometric mean of the densities of $\dga$ and $\dgak{4}$,
Cauchy--Schwarz gives
\begin{align*}
 \left|\int (\abs F^2-c)\,\Delta^2(\abs F^2)\,w_2\dd m\right|
 &\leq C\norm{\abs F^2-c}_{\Fock}
   \left(\int_{\C}\abs{F''}^4\dgak{4}\right)^{1/2}.
\end{align*}

For the second integral, we use the identity
$\abs{\nabla(\Delta(\abs F^2))}=8\abs{F'}\abs{F''}$ together with the bound
for $\abs{\nabla w_2}$. The density that occurs is the geometric mean of the
densities of $\dga$ and $\dgak{3}$, so Cauchy--Schwarz gives
\begin{align*}
 \left|2\int (\abs F^2-c)\,
   \nabla\bigl(\Delta(\abs F^2)\bigr)\cdot\nabla w_2\dd m\right|
 &\leq C\int_{\C}\abs{\abs F^2-c}\,\abs{F'}\abs{F''}
   \frac{\ee^{-\abs z^2}}{\pi(1+\abs z^2)^{3/2}}\dd m(z)\\
 &\leq C\norm{\abs F^2-c}_{\Fock}
   \left(\int_{\C}\abs{F'}^2\abs{F''}^2\dgak{3}\right)^{1/2}.
\end{align*}

For the third integral, we use the first identity in
\eqref{eq:laplacian-identities} and the bound for $\abs{\Delta w_2}$. This
time the density is the geometric mean of the densities of $\dga$ and
$\dgak{2}$, and hence
\begin{align*}
 \left|\int (\abs F^2-c)\,\Delta(\abs F^2)\,\Delta w_2\dd m\right|
 &\leq C\int_{\C}\abs{\abs F^2-c}\,\abs{F'}^2
   \frac{\ee^{-\abs z^2}}{\pi(1+\abs z^2)}\dd m(z)\\
 &\leq C\norm{\abs F^2-c}_{\Fock}
   \left(\int_{\C}\abs{F'}^4\dgak{2}\right)^{1/2}.
\end{align*}

By \eqref{eq:cauchy-fourth} and \eqref{eq:cauchy-mixed}, the last factors in
the first two estimates are also bounded by a constant multiple of
$\bigl(\int_{\C}\abs{F'}^4\dgak{2}\bigr)^{1/2}$. Summing the three estimates
in \eqref{eq:J-expansion}, we obtain
\[
 J\leq C\norm{\abs F^2-c}_{\Fock}
   \left(\int_{\C}\abs{F'}^4\dgak{2}\right)^{1/2}.
\]
Finally, \eqref{eq:j-f-prime} gives
$\bigl(\int_{\C}\abs{F'}^4\dgak{2}\bigr)^{1/2}=J^{1/2}/4$. Therefore,
\[
 J\leq C\norm{\abs F^2-c}_{\Fock}J^{1/2}.
\]
If $J=0$ there is nothing to prove; otherwise division by $J^{1/2}$ gives
\[
 J\leq C\norm{\abs F^2-c}_{\Fock}^2.
\]
Returning to \eqref{eq:first-ibp} and using the estimates for its two terms,
we obtain
\begin{align*}
 \int_{\C}\abs{\nabla(\abs F^2)}^2\dgak{1}
 &\leq C\norm{\abs F^2-c}_{\Fock}^2
   +\norm{\abs F^2-c}_{\Fock}J^{1/2}\\
 &\leq C\norm{\abs F^2-c}_{\Fock}^2.
\end{align*}
Combining this estimate with the initial application of
Lemma~\ref{lem:derivative-equivalence} gives
\[
 \norm{F^2-F(0)^2}_{\Fock}^2
 \leq C\norm{\abs F^2-c}_{\Fock}^2,
\]
for every polynomial $F$. Taking square roots and then the infimum over
$c\in\R$ proves the polynomial case of Theorem~\ref{thm:main}.

We finish the extension to the stated entire-function class. Write
$F(z)=\sum_{n\geq0}a_nz^n$ and $F^2(z)=\sum_{n\geq0}b_nz^n$. For $0\leq r<1$, set
\[
 F_r(z):=F(rz),
 \qquad P_{r,N}(z):=\sum_{n<N}a_nr^nz^n.
\]
Since
\[
 F_r^2(z)=F^2(rz)=\sum_{n\geq0}b_nr^nz^n,
\]
the coefficient formula for the Fock norm gives
\[
 \norm{F_r^2-F^2}_{\Fock}^2
 =\sum_{n\geq0}\abs{b_n}^2(1-r^n)^2n!.
\]
Since
\(\sum_{n\geq0}\abs{b_n}^2n!=\norm{F^2}_{\Fock}^2<\infty\) and
$(1-r^n)^2\leq1$, dominated convergence shows that this expression tends to
zero as $r\uparrow1$.

We next fix $0<r<1$ and prove that $P_{r,N}\to F_r$ in $L^4(\dga)$. For each
circle of radius $\rho\geq0$, Parseval's identity and Cauchy--Schwarz with
respect to the normalized angular measure give
\begin{align*}
 \left(\sum_{n\geq0}\abs{a_n}^2\rho^{2n}\right)^2
 &=\left(\frac{1}{2\pi}\int_0^{2\pi}
   \abs{F(\rho\ee^{\ii\theta})}^2\dd\theta\right)^2\\
 &\leq\frac{1}{2\pi}\int_0^{2\pi}
   \abs{F(\rho\ee^{\ii\theta})}^4\dd\theta\\
 &=\sum_{n\geq0}\abs{b_n}^2\rho^{2n}.
\end{align*}
Therefore, applying Cauchy--Schwarz to the Taylor tail, we get
\begin{align*}
 \abs{P_{r,N}(z)-F_r(z)}^4
 &\leq\left(\sum_{n\geq N}r^{2n}\right)^2
   \left(\sum_{n\geq N}\abs{a_n}^2\abs z^{2n}\right)^2\\
 &\leq\left(\sum_{n\geq N}r^{2n}\right)^2
   \sum_{n\geq0}\abs{b_n}^2\abs z^{2n}.
\end{align*}
Integrating this inequality with respect to $\dga$ and using
$\int_{\C}\abs z^{2n}\dga(z)=n!$, we obtain
\[
 \norm{P_{r,N}-F_r}_{L^4(\dga)}^4
 \leq\left(\sum_{n\geq N}r^{2n}\right)^2
   \sum_{n\geq0}\abs{b_n}^2n!
 =\left(\sum_{n\geq N}r^{2n}\right)^2\norm{F^2}_{\Fock}^2.
\]
The geometric tail tends to zero as $N\to\infty$, proving that
$P_{r,N}\to F_r$ in $L^4(\dga)$. Finally, the factorization
\[
 P_{r,N}^2-F_r^2=(P_{r,N}-F_r)(P_{r,N}+F_r)
\]
and Cauchy--Schwarz give
\[
 \norm{P_{r,N}^2-F_r^2}_{\Fock}
 \leq\norm{P_{r,N}-F_r}_{L^4(\dga)}
      \norm{P_{r,N}+F_r}_{L^4(\dga)}\longrightarrow0.
\]
Here the second factor remains bounded because $P_{r,N}\to F_r$ in
$L^4(\dga)$.

Now take $r_k=k/(k+1)$ and choose $N_k\geq1$ so that
\[
 \norm{P_{r_k,N_k}^2-F_{r_k}^2}_{\Fock}<\frac1k.
\]
Since $F_{r_k}^2\to F^2$ in $\Fock$, the triangle inequality gives
$P_{r_k,N_k}^2\to F^2$ in $\Fock$. Moreover,
$P_{r_k,N_k}(0)=F(0)$ for every $k$. For any fixed $c\in\R$, the polynomial
estimate therefore gives
\[
 \norm{P_{r_k,N_k}^2-F(0)^2}_{\Fock}
 \leq C\norm{\abs{P_{r_k,N_k}}^2-c}_{\Fock}.
\]
The pointwise inequality
$\bigl\lvert\abs{P_{r_k,N_k}}^2-\abs F^2\bigr\rvert
\leq\abs{P_{r_k,N_k}^2-F^2}$ shows that we may pass to the limit and obtain
\[
 \norm{F^2-F(0)^2}_{\Fock}
 \leq C\norm{\abs F^2-c}_{\Fock}.
\]
Taking the infimum over $c\in\R$ proves the theorem.
\end{proof}

\subsection{Proof of the square-restricted stability statement}

\begin{proof}[Proof of Corollary~\ref{cor:square-stability}]
Theorem~\ref{thm:main} with $c=1$ gives
\[
 \norm{F^2-F(0)^2}_{\Fock}
 \leq C\norm{\abs F^2-1}_{\Fock}.
\]
Since
$\abs{\abs{F}^2-\abs{F(0)}^2}\leq\abs{F^2-F(0)^2}$ pointwise and
$\dga$ is a probability measure,
\[
 \bigl|\abs{F(0)}^2-1\bigr|
 \leq \norm{\abs{F}^2-\abs{F(0)}^2}_{\Fock}
   +\norm{\abs F^2-1}_{\Fock}
 \leq (C+1)\norm{\abs F^2-1}_{\Fock}.
\]
If $F(0)\neq0$, choose
$\lambda=F(0)^2/\abs{F(0)}^2$; if $F(0)=0$, choose any unimodular
$\lambda$. In both cases,
\[
 \abs{F(0)^2-\lambda}=\bigl|\abs{F(0)}^2-1\bigr|.
\]
Consequently,
\[
 \norm{F^2-\lambda}_{\Fock}
 \leq \norm{F^2-F(0)^2}_{\Fock}+\abs{F(0)^2-\lambda}
 \leq (2C+1)\norm{\abs F^2-1}_{\Fock},
\]
which proves Corollary~\ref{cor:square-stability}.
\end{proof}

\section{Comments}\label{sec:comments}

\subsection{Comparison to \cite{Bertolini2026}}

The estimate in Corollary~\ref{cor:square-stability} has the same natural Fock-space topology and the same global form as \eqref{eq:full-result-intro}. Its scope is smaller: the competitor is required to be a square. In particular, the argument does not prove stability for an arbitrary $G\in\Fock$, and it does not address the higher Hermite windows or the true polyanalytic Fock spaces considered in \cite{Bertolini2026}.

The gain is conceptual and cosmetic, but useful. The proof is entirely complex analytic after the Bargmann--Fock reduction. It uses no circle estimate for trigonometric polynomials, no decomposition into frequency blocks, and no annular mass-concentration argument. The price of this simplification is precisely the identity
\[
 \abs{\nabla\abs F^2}^2=\abs{(F^2)' }^2,
\]
which has no counterpart for a general competitor and is not stable under the passage to polyanalytic functions.

The two proof architectures can therefore be summarized as follows. The full theorem first removes the phase direction by an abstract orthogonal reduction and then exploits polar Fourier structure. The present argument removes the constant mode by the weighted derivative equivalence, converts the modulus defect into a gradient energy, and closes the estimate by Cauchy integral bounds. The latter route is shorter, while the former route has the substantially broader conclusion.

\subsection{The Gaussian energy viewpoint and Carlen's work}

There is a natural connection with the Gaussian energy methods appearing in Carlen's work on Fisher information and logarithmic Sobolev inequalities \cite{Carlen1991}. Both arguments exploit the Gaussian measure, an integration-by-parts identity, and the fact that complex analyticity turns derivatives of a modulus into holomorphic derivatives. In the present proof, this mechanism appears in the pair of identities
\[
 \Delta\abs F^2=4\abs{F'}^2,
 \qquad
 \abs{\nabla\abs F^2}^2=\abs{(F^2)' }^2.
\]
They play the role of a particularly rigid carré-du-champ structure for the density $\abs F^2$.

In spite of the similarities, Carlen's arguments are designed around entropy, Fisher information, and sharp Gaussian inequalities, whereas the present argument is a coercivity estimate for the nonlinear squaring map. The auxiliary factor $(1+\abs z^2)^{-1}$ is inserted to make differentiation coercive away from the constant mode, and the weighted Cauchy lemma is the genuinely complex-analytic step that closes the estimate.

\section{The Lean verification}\label{sec:lean}

Every result of this note has been formalized in Lean~4 and checked by its kernel against Mathlib: Lemma~\ref{lem:derivative-equivalence}, Lemma~\ref{lem:weighted-cauchy}, Theorem~\ref{thm:main} in both its polynomial and its entire-function form, and Corollary~\ref{cor:square-stability}. The library builds with Lean \lean{v4.30.0} against Mathlib \lean{v4.30.0}; it contains no \lean{sorry}, no \lean{admit} and no project axiom, and \lean{\#print axioms} on each of the declarations listed below reports only \lean{propext}, \lean{Classical.choice} and \lean{Quot.sound}. The trusted base is therefore Lean's kernel together with the pinned Mathlib; nothing in this note is verified relative to an unproved assumption. All Lean statements reproduced or cited in this section were also read and checked by a human against the corresponding mathematical statements in the note. The code is available at \url{https://github.com/cynthiabort/SPR-near-Gaussian}.

\subsection{The verified statement}

A formalization is only as informative as the statement it proves, so we reproduce the definitions and the two theorems that carry the mathematical content. At the level of statements these depend on Mathlib alone, and reading them is enough to check that what was verified is what was claimed.

The measure is the normalized Gaussian $\dga$ of the introduction, realized as a density against planar Lebesgue measure:
\begin{leancode}
def gaussianDensity (z : ℂ) : ℝ≥0∞ :=
  ENNReal.ofReal (Real.exp (-‖z‖ ^ 2) / Real.pi)

def gaussianMeasure : Measure ℂ :=
  volume.withDensity gaussianDensity
\end{leancode}
The Fock norm is not introduced as a norm on an abstract space of entire functions, but rather, in this context it is interpreted as the Gaussian $L^2$ norm, so that the coefficient identity $\norm{H}_{\Fock}^2=\sum_n\abs{a_n}^2n!$ used throughout Section~\ref{sec:main-argument} is a proved theorem (\lean{fock\_norm\_sq\_eq\_coeff\_sum}) and not a definition.
\begin{leancode}
def functionFockNorm (f : ℂ → ℂ) : ℝ :=
  lpNorm f 2 gaussianMeasure

def functionCenteredSquare (F : ℂ → ℂ) : ℂ → ℂ :=
  fun z ↦ F z ^ 2 - F 0 ^ 2

def functionDistanceToConstant (F : ℂ → ℂ) (c : ℝ) : ℝ :=
  lpNorm (fun z ↦ Complex.normSq (F z) - c) 2 gaussianMeasure

def functionDistanceToUnimodularConstants (G : ℂ → ℂ) : ℝ :=
  sInf \lbr{}d : ℝ | ∃ phase : ℂ, ‖phase‖ = 1 ∧
    d = functionFockNorm (fun z ↦ G z - phase)\rbr{}
\end{leancode}
Here \lean{Complex.normSq (F z)} is $\abs{F(z)}^2$ and \lean{lpNorm f 2 gaussianMeasure} is $\norm{f}_{L^2(\dga)}$, so \lean{functionDistanceToConstant F c} is $\norm{\abs F^2-c}_{\Fock}$. The last definition is the infimum appearing in \eqref{eq:square-stability}, taken literally as the infimum of a set of real numbers rather than replaced by a chosen minimizing phase.

Theorem~\ref{thm:main} for entire $F$ and Corollary~\ref{cor:square-stability} are then the following two statements.
\begin{leancode}
theorem square_coercivity_for_entire :
    ∃ C : ℝ, 0 < C ∧ ∀ (F : ℂ → ℂ), Differentiable ℂ F →
      MemLp (fun z ↦ F z ^ 2) 2 gaussianMeasure →
      functionFockNorm (functionCenteredSquare F) ≤
        C * (⨅ c : ℝ, functionDistanceToConstant F c)

theorem stable_phase_retrieval_for_squares :
    ∃ C_square : ℝ, 0 < C_square ∧
      ∀ (F : ℂ → ℂ), Differentiable ℂ F →
        MemLp (fun z ↦ F z ^ 2) 2 gaussianMeasure →
        functionDistanceToUnimodularConstants (fun z ↦ F z ^ 2) ≤
          C_square * functionDistanceToConstant F 1
\end{leancode}
In these statements \lean{Differentiable ℂ F} says that $F$ is entire, \lean{MemLp (fun z ↦ F z \^{} 2) 2 gaussianMeasure} says that $F^2\in\Fock$, and \lean{⨅ c : ℝ} is the infimum over all real $c$. The constant is quantified before $F$, hence absolute, and neither statement carries any hypothesis beyond those of Theorem~\ref{thm:main} and Corollary~\ref{cor:square-stability}.

\subsection{Correspondence with the proof}

The table below records the statements from Section~\ref{sec:main-argument} whose correspondence
with one or more Lean declarations was checked by a human. For every entry, the hypotheses,
conclusion, constants, and cited Lean declaration were compared directly with the mathematical
statement in the note.

\begin{table}[htbp]
\centering\small
\begin{tabular}{@{}p{0.34\textwidth}p{0.60\textwidth}@{}}
\hline
Result in this note & Lean declaration \\
\hline
$\int_{\C}\abs z^{2k}\dga=k!$ & \lean{complex\_gaussian\_monomial\_moment} \\
$\norm{H}_{\Fock}^2=\sum_n\abs{a_n}^2n!$ & \lean{fock\_norm\_sq\_eq\_coeff\_sum} \\
\eqref{eq:kappa-comparison}, \eqref{eq:kn} & \lean{weighted\_derivative\_coefficient}\newline\lean{weighted\_derivative\_coefficient\_bounds} \\
Lemma~\ref{lem:derivative-equivalence} & \lean{entire\_weighted\_derivative\_equivalence\_integral} \\
\eqref{eq:pointwise-cauchy} & \lean{pointwise\_cauchy\_bound} \\
weight comparison under the two shifts & \lean{inner\_weight\_shift}, \lean{outer\_weight\_shift} \\
injectivity/diffeomorphism of $\Phi_\theta$ & \lean{outer\_shift\_diffeomorph} \\
Lemma~\ref{lem:weighted-cauchy} & \lean{weighted\_cauchy} \\
gradient identity in \eqref{eq:cr-identities} and first identity in \eqref{eq:laplacian-identities} & \lean{gradient\_normSq\_polynomial}\newline\lean{laplacian\_normSq\_polynomial} \\
\eqref{eq:first-ibp} & \lean{weighted\_integration\_by\_parts} \\
bounds on $\Delta w_1$, on $\nabla w_2$ and $\Delta w_2$ & \lean{weight\_laplacian\_bound}\newline\lean{weight\_two\_derivative\_bounds} \\
$J\lesssim\norm{\abs F^2-c}_{\Fock}^2$ & \lean{laplacian\_energy\_bound} \\
Theorem~\ref{thm:main}, polynomial case & \lean{squaring\_against\_every\_constant}, \lean{squaring} \\
construction of $P_{r,N}$ & \lean{entire\_fock\_square\_polynomial\_approximation} \\
Theorem~\ref{thm:main}, entire case & \lean{square\_coercivity\_for\_entire} \\
Corollary~\ref{cor:square-stability} & \lean{stable\_phase\_retrieval\_for\_squares} \\
\hline
\end{tabular}
\end{table}

\section*{Acknowledgements}

J. P. G. R. ~was supported by the FCT through project SHADE (project 2023.17881.ICDT, DOI  10.54499 / 2023.17881.ICDT), by FAPERJ through the JCNE grant no.~SEI-260003 / 020475/2025, and by Instituto Serrapilheira through grant Serra-R-2510-59765. 

We would furthermore like to express our gratitude towards Jaume de Dios Pont, Mitchell Taylor, André Guerra, and Francisco Ganacim for several discussions on the strategy to prove the theorem, the lean formalization and the overall structure of the manuscript.

\end{document}